\documentclass{scrartcl}
\usepackage{amsmath}
\usepackage{amsfonts}
\usepackage{hyperref}
\newcommand{\bbbn}{\mathbb{N}}
\newcommand{\bbbr}{\mathbb{R}}
\newcommand{\Idx}{\mathcal{I}}
\newcommand{\ctI}{\mathcal{T}_{\Idx}}
\newcommand{\lfI}{\mathcal{L}_{\Idx}}
\newcommand{\ctII}{\mathcal{T}_{\Idx\times\Idx}}
\newcommand{\lfII}{\mathcal{L}_{\Idx\times\Idx}}
\newcommand{\lfiII}{\mathcal{L}_{\Idx\times\Idx}^-}
\newcommand{\lfaII}{\mathcal{L}_{\Idx\times\Idx}^+}

\newcommand{\ctIII}{\mathcal{T}_{\Idx\times\Idx\times\Idx}}
\newcommand{\sons}{\mathop{\operatorname{sons}}\nolimits}
\newcommand{\desc}{\mathop{\operatorname{desc}}\nolimits}
\newcommand{\level}{\mathop{\operatorname{level}}\nolimits}
\newcommand{\rnk}{\mathop{\operatorname{rank}}\nolimits}

\newcommand{\Csp}{C_\text{sp}}
\newcommand{\Cad}{C_\text{ad}}
\newcommand{\Cmg}{C_\text{mg}}
\newcommand{\Cup}{C_\text{up}}
\newcommand{\Cmm}{C_\text{mm}}
\newcommand{\Wev}{W_\text{ev}}
\newcommand{\Wup}{W_\text{up}}
\newcommand{\Wmm}{W_\text{mm}}
\newcommand{\Wls}{W_\text{ls}}
\newcommand{\Wrs}{W_\text{rs}}
\newcommand{\Wll}{W_\text{ll}}
\newcommand{\Wrl}{W_\text{rl}}
\newcommand{\Wlr}{W_\text{lr}}
\newcommand{\Wrr}{W_\text{rr}}
\newcommand{\Wdc}{W_\text{dc}}
\newcommand{\Wli}{W_\text{li}}
\newcommand{\Wri}{W_\text{ri}}
\newcommand{\Win}{W_\text{in}}
\newtheorem{theorem}{Theorem}
\newtheorem{lemma}[theorem]{Lemma}
\newtheorem{definition}[theorem]{Definition}
\newtheorem{remark}[theorem]{Remark}
\newenvironment{proof}{\emph{Proof.}}{$\Box$}
\allowdisplaybreaks
\title{Complexity estimates for triangular hierarchical matrix
  algorithms}
\author{Steffen B\"orm}
\begin{document}
\maketitle
\begin{abstract}
  Triangular factorizations are an important tool for solving
  integral equations and partial differential equations with
  hierarchical matrices ($\mathcal{H}$-ma\-trices).

  Experiments show that using an $\mathcal{H}$-matrix LR factorization
  to solve a system of linear questions is superior to direct inversion
  both with respect to accuracy and efficiency, but so far theoretical
  estimates quantifying these advantages were missing.

  Due to a lack of symmetry in $\mathcal{H}$-matrix algorithms, we cannot
  hope to prove that the LR factorization takes one third of the operations
  of the inversion or the matrix multiplication, as in standard linear algebra.
  We can, however, prove that the LR factorization together with two
  other operations of similar complexity, i.e., the inversion and
  multiplication of triangular matrices, requires not more operations
  than the matrix multiplication.

  We can complete the estimates by proving an improved upper bound for the
  complexity of the matrix multiplication, designed for recently introduced
  variants of classical $\mathcal{H}$-matrices.
\end{abstract}

\section{Introduction}

Hierarchical matrices \cite{HA99,GRHA02}, $\mathcal{H}$-matrices for
short, can be used to approximate certain densely populated matrices
arising in the context of integral equations
\cite{BE00a,BOGR02,BOGR04} and elliptic partial
differential equations \cite{BEHA03,FAMEPR13} in
linear-polylogarithmic complexity.
Compared to other methods like fast multipole expansions
\cite{GRRO87,GRRO97} or wavelet approximations \cite{BECORO91,CODADE01},
it is possible to approximate arithmetic operations like the
matrix multiplication, inversion, or triangular factorization
for $\mathcal{H}$-matrices in linear-polylogarithmic complexity.
This property makes $\mathcal{H}$-matrices attractive for
a variety of applications, starting with solving partial
differential equations \cite{BEHA03,FAMEPR13} and
integral equations \cite{FAMEPR13b}, up to dealing with matrix equations
\cite{GR01a,GRHAKH02,BABE06,BA08} and evaluating matrix functions
\cite{GAHAKH00,GAHAKH02}.

Already the first articles on $\mathcal{H}$-matrix techniques
introduced an algorithm for approximating the inverse of an
$\mathcal{H}$-matrix by recursively applying a block representation
\cite{HA99,GRHA02}.
This approach works well, but is quite time-consuming.

The situation improved significantly when Lintner and Grasedyck
introduced an efficient algorithm for approximating the LR factorization
of an $\mathcal{H}$-matrix \cite{LI04,GRKRLE05a}, reducing the
computational work by a large factor and simultaneously
considerably improving the accuracy.
It is fairly easy to prove that the $\mathcal{H}$-LR or
$\mathcal{H}$-Cholesky factorization requires less computational
work than the $\mathcal{H}$-matrix multiplication or inversion,
and for the latter operations linear-polylogarithmic complexity
bounds have been known for years \cite{GRHA02}.

For dense $n\times n$ matrices in standard array representation, we know
that a straightforward implementation of the LR factorization
requires
\begin{subequations}\label{eq:dense_operations}
\begin{align}
  \sum_{\ell=1}^n (n-\ell) + 2 (n-\ell)^2
  &= \frac{n}{6} (4 n^2 - 3 n - 1) \quad\text{operations}
     \label{eq:dense_lr},
\intertext{i.e., approximately one third of the $2n^3$ operations
required for the matrix multiplication.
We would like to prove a similar result for $\mathcal{H}$-matrices,
but this is generally impossible due to the additional steps
required to obtain low-rank approximations of intermediate
results.
We can circumvent this problem by considering two further
operations:
the inversion of a triangular matrix requires}
  \sum_{\ell=1}^n 1 + (n-\ell) + (n-\ell)^2
  &= \frac{n}{6}(2 n^2 + 4) \quad\text{operations}
     \label{eq:dense_linv},
\intertext{i.e., approximately one sixth of the operations
required for the matrix multiplication, while multiplying an
upper and a unit-diagonal lower triangular matrix takes}
  \sum_{\ell=1}^n (n-\ell) + 2 (n-\ell)^2
  &= \frac{n}{6}(4 n^2 - 3 n - 1) \quad\text{operations.}
     \label{eq:dense_rl}
\end{align}
This means that computing the LR factorization $G=LR$ of a
matrix $G\in\bbbr^{n\times n}$, inverting $L$ and $R$, and
multiplying them to obtain the inverse $G^{-1} = R^{-1} L^{-1}$
requires a total of
\begin{equation*}
  \frac{n}{6} (12 n^2 - 6 n + 6) \leq 2 n^3
  \quad\text{operations},
\end{equation*}
i.e., the three operations \emph{together} require approximately
as much work as the matrix multiplication.
\end{subequations}

Our first goal is to prove that this statement also holds for
$\mathcal{H}$-matrices with (almost) \emph{arbitrary} block trees,
i.e., that the operations appearing in the factorization,
triangular inversion, and multiplication fit together like the parts
of a jigsaw puzzle corresponding to the $\mathcal{H}$-matrix
multiplication.
Incidentally, combining the three algorithms also allows us to
compute the approximate $\mathcal{H}$-matrix inverse in place
without the need for a separate output matrix.

In order to complete the complexity analysis, we also have to
show that the $\mathcal{H}$-matrix multiplication has linear-polylogarithmic
complexity.
This is already known \cite{GRHA02,HA15}, but can find an
improved estimate that reduces the impact of the sparsity
of the block tree and therefore may be interesting for recently
developed versions of $\mathcal{H}$-matrices, e.g., MBLR-matrices,
that use a denser block tree to improve the potential for parallelization
\cite{AMASBOBUEXWE15,AMBUEXMA17}.

\section{Definitions}

The blockwise low-rank structure of $\mathcal{H}$-matrices
$G\in\bbbr^{\Idx\times\Idx}$ is
conveniently described by the \emph{cluster tree}, a hierarchical
subdivision of an index set $\Idx$ into disjoint subsets $\hat t$
called \emph{clusters}, and the \emph{block tree}, a hierarchical
subdivision of a product index set $\Idx\times\Idx$ into subsets
$\hat t\times\hat s$ constructed from these clusters.

%
%
\begin{definition}[Cluster tree]
\label{de:cluster_tree}
Let $\Idx$ be a finite index set.
A tree $\ctI$ is a \emph{cluster tree} for this index set if
each node $t\in\ctI$ is labeled with a subset $\hat t\subseteq\Idx$
and if these subsets satisfy the following conditions:
\begin{itemize}
  \item The root of $\ctI$ is labeled with $\Idx$.
  \item If $t\in\ctI$ has sons, the label of $t$ is the union
    of the labels of the sons, i.e.,
    $\hat t = \bigcup_{t'\in\sons(t)} \hat t'$.
  \item The labels of sons of $t\in\ctI$ are disjoint, i.e.,
    for $t\in\ctI$ and $t_1,t_2\in\sons(t)$ with $t_1\neq t_2$,
    we have $\hat t_1\cap\hat t_2=\emptyset$.
\end{itemize}
The nodes of a cluster tree are called \emph{clusters}.
The set of leaves is denoted by $\lfI$.
\end{definition}

%
%
\begin{definition}[Block tree]
\label{de:block_tree}
Let $\ctI$ be a cluster tree for an index set $\Idx$.
A tree $\ctII$ is a \emph{block tree} for this cluster tree if
\begin{itemize}
  \item For each node $b\in\ctII$ there are cluster $t,s\in\ctI$
    with $b=(t,s)$.
    $t$ is called the \emph{row cluster} for $b$ and
    $s$ the \emph{column cluster}.
  \item If $r\in\ctI$ is the root of $\ctI$, the root of $\ctII$
    is $b=(r,r)$.
  \item If $b=(t,s)\in\ctII$ has sons, they are pairs of the
    sons of $t$ and $s$, i.e., $\sons(b)=\sons(t)\times\sons(s)$.
\end{itemize}
The nodes of a block tree are called \emph{blocks}.
The set of leaves is denoted by $\lfII$.
\end{definition}

We can see that the labels of the leaves of a cluster tree $\ctI$
correspond to a disjoint partition of the index set $\Idx$ and
that the sets $\hat t\times\hat s$ with $(t,s)\in\lfII$ correspond
to a disjoint partition of the index set $\Idx\times\Idx$, i.e.,
of a decomposition of a matrix $G\in\bbbr^{\Idx\times\Idx}$ into
submatrices.

Among the leaf blocks $\lfII$, we identify those that correspond
to submatrices that we expect to have low numerical rank.
These blocks are called \emph{admissible} and collected in the
set $\lfaII\subseteq\lfII$.
The remaining leaves are called \emph{inadmissible} and
collected in the set $\lfiII := \lfII \setminus \lfaII$.

We note that there are efficient algorithms at our disposal
for constructing cluster and block trees for various
applications \cite{GRHA02,HA15}.

%
%
\begin{definition}[Hierarchical matrix]
Let $\ctI$ be a cluster tree for an index set $\Idx$, and
let $\ctII$ be a block tree for $\ctI$ with sets $\lfaII$ and
$\lfiII$ of admissible and inadmissible leaves.
A matrix $G\in\bbbr^{\Idx\times\Idx}$ is a \emph{hierarchical
matrix} (or short \emph{$\mathcal{H}$-matrix}) of local rank
$k\in\bbbn$ if
\begin{align*}
  \rnk(G|_{\hat t\times\hat s}) &\leq k &
  &\text{ holds for all } b=(t,s)\in\lfaII,
\end{align*}
i.e., if all admissible leaves have a rank smaller or equal
to $k$.
\end{definition}

If $G$ is a hierarchical matrix, we can find matrices
$A_{ts}\in\bbbr^{\hat t\times k}$ and $B_{ts}\in\bbbr^{\hat s\times k}$
for every admissible leaf $b=(t,s)\in\lfaII$ such that
\begin{equation}\label{eq:lowrank_factorized}
  G|_{\hat t\times\hat s} = A_{ts} B_{ts}^*.
\end{equation}
Here we use the shorthand notation $\bbbr^{\hat t\times k}$ for
the set $\bbbr^{\hat t\times[1:k]}$ of matrices with row
indices in $\hat t\subseteq\Idx$ and column indices in $[1:k]$.

For inadmissible leaves $b=(t,s)\in\lfiII$, we store the
\emph{nearfield matrices} $N_{ts}\in\bbbr^{\hat t\times\hat s}$
directly.
The matrix families $(A_{ts})_{(t,s)\in\lfaII}$,
$(B_{ts})_{(t,s)\in\lfaII}$, and $(N_{ts})_{(t,s)\in\lfiII}$
together represent an $\mathcal{H}$-matrix
$G\in\bbbr^{\Idx\times\Idx}$.

\section{\texorpdfstring{Basic $\mathcal{H}$-matrix operations}%
                        {Basic H-matrix operations}}

Before we consider algorithms for triangular $\mathcal{H}$-matrices,
we have to recall the algorithms they are based on:
the $\mathcal{H}$-matrix-vector multiplication, the
$\mathcal{H}$-matrix low-rank update, and the $\mathcal{H}$-matrix
multiplication.

The multiplication an $\mathcal{H}$-matrix $G\in\bbbr^{\Idx\times\Idx}$
with multiple vectors collected in the columns of a matrix
$Y\in\bbbr^{\Idx\times\ell}$ can be split into updates
\begin{align}\label{eq:addeval}
  X|_{\hat t} &\gets X|_{\hat t} + \alpha G|_{\hat t\times\hat s} Y|_{\hat s} &
  &\text{ for } b=(t,s)\in\ctII,
\end{align}
where $X|_{\hat t} := X|_{\hat t\times \ell}$ denotes the restriction of $X$
to the row indices in $\hat t$ and $\alpha\in\bbbr$ is a scaling
factor.
For inadmissible leaves, the update can be carried out directly,
taking care to minimize the computational work by ensuring that the
scaling by $\alpha$ is applied to $Y|_{\hat s}$ if $|\hat s|\leq|\hat t|$
and to the product $G|_{\hat t\times\hat s} Y|_{\hat s}$ otherwise.

For admissible leaves we have $G|_{\hat t\times\hat s}
= A_{ts} B_{ts}^*$ and can first compute the intermediate
matrix $\widehat{Z}_{ts} := \alpha B_{ts}^* Y|_{\hat s}\in\bbbr^{k\times\ell}$
and then add $A_{ts} \widehat{Z}_{ts}$ to the output, i.e.,
\begin{align*}
  X|_{\hat t} &\gets X|_{\hat t} + \alpha A_{ts} B_{ts}^* Y|_{\hat s}
    = X|_{\hat t} + A_{ts} \widehat{Z}_{ts} &
  &\text{ for all } b=(t,s)\in\lfaII.
\end{align*}
For non-leaf blocks $b=(t,s)\in\ctII\setminus\lfII$, we recursively
consider sons until we arrive at leaves.
In total, the number of operations for (\ref{eq:addeval}) is equal to
\begin{align}\label{eq:Wev}
  \Wev(t,s,\ell)
  &:= \begin{cases}
        2 \ell k (|\hat t| + |\hat s|)
        &\text{ if } (t,s)\in\lfaII,\\
        \ell (2 |\hat t|\,|\hat s| + \min\{|\hat t|,|\hat s|\})
        &\text{ if } (t,s)\in\lfiII,\\
        \sum_{\substack{t'\in\sons(t)\\s'\in\sons(s)}} \Wev(t',s',\ell)
        &\text{ otherwise}
      \end{cases}
\end{align}
for all $b=(t,s)\in\ctII$.
The multiplication by the transposed $\mathcal{H}$-matrix $G^*$,
i.e., updates of the form
\begin{align}\label{eq:addevaltrans}
  X|_{\hat s} &\gets X|_{\hat s} + \alpha G|_{\hat t\times\hat s}^* Y|_{\hat t} &
  &\text{ for } b=(t,s)\in\ctII,
\end{align}
can be handled simultaneously and also requires $\Wev(t,s,\ell)$ operations.
In the following, we assume that procedures ``addeval'' and
``addevaltrans'' for the operations (\ref{eq:addeval}) and
(\ref{eq:addevaltrans}) are at our disposal.

The low-rank update of an $\mathcal{H}$-matrix $G$, i.e., the
approximation of
\begin{equation}\label{eq:update}
  G|_{\hat t\times\hat s}
  \gets G|_{\hat t\times\hat s} + A B^*
\end{equation}
for a block $(t,s)\in\ctII$, $A\in\bbbr^{\hat t\times\ell}$ and
$B\in\bbbr^{\hat s\times\ell}$ is realized by recursively
moving to the leaves of the block tree and performing a direct
update for inadmissible leaves and a truncated update for
admissible ones:
if $(t,s)\in\lfaII$, we have $G|_{\hat t\times\hat s} = A_{ts} B_{ts}^*$
and approximate
\begin{align*}
  G|_{\hat t\times\hat s}
  &\gets G|_{\hat t\times\hat s} + A B^*
   = \widehat{A} \widehat{B}^*, &
  \widehat{A} &:= \begin{pmatrix} A_{ts} & A \end{pmatrix}, &
  \widehat{B} &:= \begin{pmatrix} B_{ts} & B \end{pmatrix}
\end{align*}
by computing the thin Householder factorization
$\widehat{B} = Q R$, a low-rank approximation $C \widehat{D}^*$ of
$\widehat{A} R^*$, so that $D:=Q \widehat{D}$ yields a low-rank
approximation $C D^* = C \widehat{D}^* Q^* \approx \widehat{A} R^* Q^*
= \widehat{A} \widehat{B}^*$.
Assuming that the Householder factorization and the low-rank
approximation of $n\times m$ matrices require $\mathcal{O}(n m \min\{n,m\})$
operations, we find a constant $\Cad$ such that the number
of operations for a low-rank update (\ref{eq:update}) is bounded by
\begin{align}\label{eq:Wup}
  \Wup(t,s,\ell) &:= \begin{cases}
    \Cad (k+\ell)^2 (|\hat t|+|\hat s|)
    &\text{ if } (t,s)\in\lfaII,\\
    2 \ell |\hat t|\,|\hat s|
    &\text{ if } (t,s)\in\lfiII,\\
    \sum_{\substack{t'\in\sons(t)\\s'\in\sons(s)}}
    \Wup(t',s',\ell)
    &\text{ otherwise}
  \end{cases}
\end{align}
for all $b=(t,s)\in\ctII$.
In the following, we assume that a procedure ``update'' for
approximating the operation (\ref{eq:update}) in this way is available. 

During the course of the $\mathcal{H}$-matrix multiplication, we may
have to split a low-rank matrix into submatrices, perform updates to
these submatrices, and then merge them into a larger low-rank matrix.
This task can be handled essentially like the update, but we have to take
special care in case that the number of submatrices is large.
Let $t,s\in\ctII$ with $|\sons(s)| = m\in\bbbn$ and
$\sons(s)=\{s_1,\ldots,s_m\}$.
We are looking for an approximation of the matrix
\begin{align*}
  G|_{\hat t\times\hat s}
  &= \begin{pmatrix}
       A_1 B_1^* & \cdots & A_m B_m^*
     \end{pmatrix}, &
  A_j\in\bbbr^{\hat t\times k},\ B_j\in\bbbr^{\hat s_j\times k}
    \text{ for all } i,j\in[1:m],
\end{align*}
where we assume that preliminary compression steps have ensured
$k\leq |\hat t|$.

In a first step, we compute thin Householder factorizations
$B_j = Q_j R_j$ with $R_j\in\bbbr^{k\times k}$ for all $j\in[1:m]$.
Due to our assumption and $|\hat s|=|\hat s_1|+\ldots+|\hat s_m|$,
this requires $\mathcal{O}(|\hat s| k^2)$ operations.
Now we form the reduced matrix
\begin{equation*}
  \widehat{G}
  := \begin{pmatrix}
       A_1 R_1^* & \cdots & A_m R_m^*
     \end{pmatrix} \in \bbbr^{\hat t\times (km)}.
\end{equation*}
Once we have found a rank-$k$ approximation $\widehat{A} \widehat{Q}^*
\approx \widehat{G}$ with $\widehat{A}\in\bbbr^{\hat t\times k}$
and an isometric matrix $\widehat{Q}\in\bbbr^{(mk)\times k}$, applying
the Householder reflections yields the rank-$k$ approximation
\begin{equation*}
  G|_{\hat t\times\hat s}
  = \widehat{G}
    \begin{pmatrix}
      Q_1^* & & \\
      & \ddots & \\
      & & Q_m^*
    \end{pmatrix}
  \approx \widehat{A} \widehat{Q}^*
    \begin{pmatrix}
      Q_1^* & & \\
      & \ddots & \\
      & & Q_m^*
    \end{pmatrix}
\end{equation*}
of the original matrix in $\mathcal{O}(|\hat s| k^2)$ operations.
To construct a rank-$k$ approximation of $\widehat{G}$, we can
proceed sequentially:
first we use techniques like the singular value decomposition or
rank-revealing QR factorization to obtain a rank-$k$ approximation
\begin{equation*}
  \widehat{A}_{m-1} \widehat{Q}_{m-1}^*
  \approx \begin{pmatrix}
            A_{m-1} R_{m-1}^* & A_m R_m^*
          \end{pmatrix}
\end{equation*}
with $\widehat{A}_{m-1}\in\bbbr^{\hat t\times k}$ and an isometric
matrix $\widehat{Q}_{m-1}\in\bbbr^{(2k)\times k}$.
Due to our assumptions, this task can be accomplished in
$\mathcal{O}(|\hat t| k^2)$ operations.
We find
\begin{align*}
  \widehat{G}
  &= \begin{pmatrix}
      A_1 R_1^* & \cdots & A_{m-1} R_{m-1}^* & A_m R_m^*
    \end{pmatrix}
   \approx \begin{pmatrix}
      A_1 R_1^* & \cdots & \widehat{A}_{m-1} \widehat{Q}_{m-1}^*
    \end{pmatrix}\\
  &= \begin{pmatrix}
      A_1 R_1^* & \cdots & \widehat{A}_{m-1}
    \end{pmatrix}
    \begin{pmatrix}
      I_{(m-2)k} & \\
      & \widehat{Q}_{m-1}^*
    \end{pmatrix},
\end{align*}
where $I_{(m-2)k}$ denotes the $(m-2)k$-dimensional identity matrix.
The left factor now has only $(m-1)k$ columns, and repeating the
procedure $m-2$ times yields
\begin{equation*}
  \widehat{G}
  \approx \widehat{A}_1
    \widehat{Q}_1^*
    \begin{pmatrix}
      I_k & \\
      & \widehat{Q}_2^*
    \end{pmatrix}
    \begin{pmatrix}
      I_{2k} & \\
      & \widehat{Q}_3^*
    \end{pmatrix}
    \cdots
    \begin{pmatrix}
      I_{(m-2)k} & \\
      & \widehat{Q}_{m-1}^*
    \end{pmatrix}.
\end{equation*}
Since the matrices $\widehat{Q}_1,\ldots,\widehat{Q}_{m-1}$ have
only $k$ columns and $2k$ rows by construction, we can compute
\begin{equation*}
  \widehat{Q}
  := \begin{pmatrix}
       I_{(m-2)k} & \\
       & \widehat{Q}_{m-1}
     \end{pmatrix}
     \begin{pmatrix}
       I_{(m-3)k} & \\
       & \widehat{Q}_{m-2}
     \end{pmatrix}
     \cdots
     \begin{pmatrix}
       I_k & \\
       & \widehat{Q}_2
     \end{pmatrix}
     \widehat{Q}_1
\end{equation*}
in $\mathcal{O}(k^3 (m-1))\subseteq\mathcal{O}(|\hat t| k^2 (m-1))$ operations
and find the desired low-rank approximation.
We conclude that there is a constant $\Cmg'$ such that not
more than $\Cmg' k^2 (|\hat t| (m-1) + |\hat s|)$
operations are needed to merge the row blocks.
We can apply the same procedure to merge column blocks, as well,
and see that
\begin{equation*}
  \Cmg' k^2 (|\hat t|\,|\sons(s)| + |\hat s| (2 |\sons(t)| - 1))
  \leq \Cmg k^2 \sum_{\substack{t'\in\sons(t)\\ s'\in\sons(s)}}
             |\hat t'| + |\hat s'|
\end{equation*}
operations are sufficient to merge submatrices for all the sons of
a block $(t,s)\in\ctI\times\ctI$, where $\Cmg := 2 \Cmg$.
In the following, we assume that a procedure ``merge'' for this task is
available.

Finally, the $\mathcal{H}$-matrix multiplication algorithm
carries out the approximate update
$Z|_{\hat t\times\hat r} \gets Z|_{\hat t\times\hat r}
+ \alpha X|_{\hat t\times\hat s} Y|_{\hat s\times\hat r}$ with
$(t,s)\in\ctII$ and $(s,r)\in\ctII$, again by recursively
considering sons of the blocks until one of them is a leaf,
so the product $X|_{\hat t\times\hat s} Y|_{\hat s\times\hat r}$ is
of low rank and can be computed using the functions
``addeval'' and ``addevaltrans''.
The functions ``update'' and ``merge'' can then be used to add the
result to $Z|_{\hat t\times\hat r}$, performing low-rank truncations
if necessary.
The algorithm is summarized in Figure~\ref{fi:addmul}.

%
%
\begin{figure}
  \begin{tabbing}
    \textbf{procedure} addmul($\alpha$, $t$, $s$, $r$,
       $X$, $Y$, \textbf{var} $Z$);\\
    \textbf{if} $(t,s)\in\lfaII$ \textbf{then begin}\\
    \quad\= $\widehat{A} \gets A_{X,ts}$;
      \quad $\widehat{B} \gets 0$;
      \quad addevaltrans($\bar\alpha$, $s$, $r$, $Y$, $B_{X,ts}$,
                                  $\widehat{B}$);\\
    \> update($t$, $r$, $\widehat{A}$, $\widehat{B}$, $Z$)\\
    \textbf{end else if} $(t,s)\in\lfiII$ \textbf{then begin}\\
    \> \textbf{if} $|\hat t|\leq |\hat s|$ \textbf{then begin}\\
    \> \quad\= $\widehat{A} \gets I_{\hat t\times\hat t}$;
      \quad $\widehat{B} \gets 0$;
      \quad addevaltrans($\bar\alpha$, $s$, $r$, $Y$, $N_{X,ts}^*$,
                                  $\widehat{B}$)\\
    \> \textbf{end else begin}\\
    \> \> $\widehat{A} \gets N_{X,ts}$;
      \quad $\widehat{B} \gets 0$;
      \quad addevaltrans($\bar\alpha$, $s$, $r$, $Y$, $I_{\hat s\times\hat s}$,
                                  $\widehat{B}$)\\
    \> \textbf{end};\\
    \> update($t$, $r$, $\widehat{A}$, $\widehat{B}$, $Z$)\\
    \textbf{end else if} $(s,r)\in\lfaII$ \textbf{then begin}\\
    \> $\widehat{B} \gets B_{Y,sr}$;
      \quad $\widehat{A} \gets 0$;
      \quad addeval($\alpha$, $t$, $s$, $X$, $A_{Y,sr}$,
                             $\widehat{A}$);\\
    \> update($t$, $r$, $\widehat{A}$, $\widehat{B}$, $Z$)\\
    \textbf{end else if} $(s,r)\in\lfiII$ \textbf{then begin}\\
    \> \textbf{if} $|\hat r|\leq|\hat s|$ \textbf{then begin}\\
    \> \> $\widehat{B} \gets I_{\hat r\times\hat r}$;
      \quad $\widehat{A} \gets 0$;
      \quad addeval($\alpha$, $t$, $s$, $X$, $N_{Y,sr}$,
                             $\widehat{A}$)\\
    \> \textbf{end else begin}\\
    \> \> $\widehat{B} \gets N_{Y,sr}^*$;
      \quad $\widehat{A} \gets 0$;
      \quad addeval($\alpha$, $t$, $s$, $X$, $I_{\hat s\times\hat s}$,
                             $\widehat{A}$)\\
    \> \textbf{end};\\
    \> update($t$, $r$, $\widehat{A}$, $\widehat{B}$, $Z$)\\
    \textbf{else begin}\\
    \> \textbf{for} $t'\in\sons(t)$, $s'\in\sons(s)$,
                    $r'\in\sons(r)$ \textbf{do}\\
    \> \> addmul($\alpha$, $t'$, $s'$, $r'$,
                          $X$, $Y$, $Z$);\\
    \> \textbf{if} $(t,r)\not\in\ctII\setminus\lfII$ \textbf{then}\\
    \> \> merge($t$, $r$, $Z$)\\
    \textbf{end}
  \end{tabbing}

  \caption{$\mathcal{H}$-matrix multiplication
           $Z|_{\hat t\times\hat r} \gets Z|_{\hat t\times\hat r}
            + \alpha X|_{\hat t\times\hat s} Y|_{\hat s\times\hat r}$}
  \label{fi:addmul}
\end{figure}

In order to keep the notation short, we introduce the
\emph{local rank} of leaf blocks by
\begin{align*}
  k_{ts} &:= \begin{cases}
    k &\text{ if } b=(t,s)\in\lfaII,\\
    \min\{|\hat t|,|\hat s|\} &\text{ otherwise}
  \end{cases} &
  &\text{ for all } b=(t,s)\in\lfII.
\end{align*}
We can see that the multiplication algorithm in Figure~\ref{fi:addmul}
performs matrix-vector multiplications for matrices with $k_{ts}$
columns if $(t,s)\in\lfII$ and for matrices with $k_{sr}$ columns
if $(s,r)\in\lfII$, followed by a low-rank update.

We obtain the bound
\begin{align}\label{eq:Wmm}
  \Wmm(t,s,r) &:= \begin{cases}
    \Wev(s,r,k_{ts}) + \Wup(t,r,k_{ts})
    &\text{ if } (t,s)\in\lfII,\\
    \Wev(t,s,k_{sr}) + \Wup(t,r,k_{sr})
    &\text{ if } (t,s)\not\in\lfII,(s,r)\in\lfII,\\
    \sum_{\substack{t'\in\sons(t),\\s'\in\sons(s),\\r'\in\sons(r)}}
      \Wmm(t',s',r')
    &\text{ otherwise}\\
    \quad + \Cmg k^2 (|\hat t|+|\hat r|)
  \end{cases}
\end{align}
for the computational work required by the algorithm
``addmul'' in Figure~\ref{fi:addmul} called with
$(t,s),(s,r)\in\ctII$, where we include the work for merging
submatrices in all non-leaf cases for the sake of simplicity.


\section{\texorpdfstring{Algorithms for triangular $\mathcal{H}$-matrices}
                        {Algorithms for triangular H-matrices}}

In the context of algorithms for hierarchical matrices, we assume
triangular matrices to be compatible with the structure of the
cluster tree $\ctI$, i.e., if we have two sons $t_1,t_2\in\sons(t)$
of a cluster $t\in\ctI$ and if there are indices $i\in\hat t_1$
and $j\in\hat t_2$ with $i<j$, \emph{all} indices in $\hat t_1$ are
smaller than \emph{all} indices in $\hat t_2$.
For this property, we use the shorthand $t_1<t_2$.

While this may appear to be a significant restriction at first
glance, it rarely poses problems in practice, since we can
\emph{define} the order of the indices in the index set $\Idx$
to satisfy our condition by simply choosing an arbitrary order
on the indices in leaf clusters and an arbitrary order on the
sons of non-leaf clusters.
By induction, these orders give rise to a global order on $\Idx$
satisfying our requirements.

We are mainly interested in three operations:
the construction of an LR factorization $G=LR$, i.e., the decomposition
of $G$ into a left lower triangular matrix with unit diagonal $L$ and a
right upper triangular matrix $R$, the inversion of the triangular
matrices, and the multiplication of triangular matrices.
Together, these three operations allow us to overwrite a matrix with
its inverse.

For the sake of simplicity, we assume that we are working with a
binary cluster tree $\ctI$, i.e., a cluster is either a leaf or has
exactly two sons $t_1<t_2$.
In the latter case, we can split triangular matrices into
submatrices
\begin{align}\label{eq:submatrices_LR}
  L_{\nu\mu} &:= L|_{\hat t_\nu\times\hat t_\mu}, &
  R_{\nu\mu} &:= R|_{\hat t_\nu\times\hat t_\mu} &
  &\text{ for all } \nu,\mu\in\{1,2\}
\end{align}
to obtain
\begin{align*}
  L|_{\hat t\times\hat t}
  &= \begin{pmatrix}
       L_{11} & \\
       L_{21} & L_{22}
     \end{pmatrix}, &
  R|_{\hat t\times\hat t}
  &= \begin{pmatrix}
       R_{11} & R_{12}\\
       & R_{22}
     \end{pmatrix}.
\end{align*}
We also assume that diagonal blocks, i.e., blocks of the form
$b=(t,t)$ with $t\in\ctI$, are never admissible.

%
%

\paragraph{Forward and backward substitution}
In order to use an LR factorization as a solver, we have to be
able to solve systems $LX=Y$, $RX=Y$, $XL=Y$, and $XR=Y$.
The third equation can be reduced to the second by taking the
adjoint, and the fourth equation similarly reduces to the first.
We consider the more general tasks of solving
\begin{align*}
  L|_{\hat t\times\hat t} X|_{\hat t} &= Y|_{\hat t}, &
  R|_{\hat t\times\hat t} X|_{\hat t} &= Y|_{\hat t}, &
  L|_{\hat t\times\hat t}^* X|_{\hat t} &= Y|_{\hat t}, &
  R|_{\hat t\times\hat t}^* X|_{\hat t} &= Y|_{\hat t}
\end{align*}
for an arbitrary cluster $t\in\ctI$.

We first address the case that $X$ and $Y$ are matrices in
standard representation, i.e., that no low-rank approximations
are required.

If $t$ is a leaf, $L|_{\hat t\times\hat t}$ and $R|_{\hat t\times\hat t}$
are a standard matrices and we can solve the equations by standard
forward and backward substitution.

If $t$ is not a leaf, the first equation takes the form
\begin{equation*}
  \begin{pmatrix}
    Y|_{\hat t_1}\\
    Y|_{\hat t_2}
  \end{pmatrix}
  = Y|_{\hat t}
  = L|_{\hat t\times\hat t} X|_{\hat t}
  = \begin{pmatrix}
    L_{11} & \\
    L_{21} & L_{22}
  \end{pmatrix}
  \begin{pmatrix}
    X|_{\hat t_1}\\
    X|_{\hat t_2}
  \end{pmatrix}
  = \begin{pmatrix}
    L_{11} X|_{\hat t_1}\\
    L_{21} X|_{\hat t_1} + L_{22} X|_{\hat t_2}
  \end{pmatrix},
\end{equation*}
so we can solve $L_{11} X|_{\hat t_1} = Y|_{\hat t_1}$ by recursion,
overwrite $Y|_{\hat t_2}$ by $\widetilde{Y}_2 := Y|_{\hat t_2}
- L_{21} X|_{\hat t_1}$ using the algorithm ``addeval'', and
solve $L_{22} X|_{\hat t_2} = \widetilde{Y}_2$ by recursion.

The second equation takes the form
\begin{equation*}
  \begin{pmatrix}
    Y|_{\hat t_1}\\
    Y|_{\hat t_2}
  \end{pmatrix}
  = Y|_{\hat t}
  = R|_{\hat t\times\hat t} X|_{\hat t}
  = \begin{pmatrix}
    R_{11} & R_{12}\\
    & R_{22}
  \end{pmatrix}
  \begin{pmatrix}
    X|_{\hat t_1}\\
    X|_{\hat t_2}
  \end{pmatrix}
  = \begin{pmatrix}
    R_{11} X|_{\hat t_1} + R_{12} X|_{\hat t_2}\\
    R_{22} X|_{\hat t_2}
  \end{pmatrix},
\end{equation*}
so we can solve $R_{22} X|_{\hat t_2} = Y|_{\hat t_2}$ by recursion,
overwrite $Y|_{\hat t_1}$ by $\widetilde{Y}_1 := Y|_{\hat t_1}
- R_{12} X|_{\hat t_2}$ using the algorithm ``addeval'' again,
and solve $R_{11} X|_{\hat t_1} = \widetilde{Y}_1$ by recursion.
The algorithms are summarized in Figure~\ref{fi:lsolve}.
Counterparts ``lsolvetrans'' and ``rsolvetrans''
for the adjoint matrices $L^*$ and $R^*$ can be defined in a
similar fashion using ``addevaltrans'' instead of ``addeval''.

%
%
\begin{figure}
  \begin{minipage}{0.45\textwidth}
  \begin{tabbing}
    \textbf{procedure} lsolve($t$, $L$,
       \textbf{var} $Y$, $X$);\\
    \textbf{if} $\sons(t)=\emptyset$ \textbf{then}\\
    \quad\= Solve $L|_{\hat t\times\hat t} X|_{\hat t} = Y|_{\hat t}$
        directly\\
    \textbf{else begin}\\
    \> lsolve($t_1$, $L$, $Y$, $X$);\\
    \> addeval($-1$, $t_2$, $t_1$, $L$, $Y$, $X$);\\
    \> lsolve($t_2$, $L$, $Y$, $X$)\\
    \textbf{end}
  \end{tabbing}
  \hfill\strut
  \end{minipage}%
  \quad%
  \begin{minipage}{0.45\textwidth}
  \begin{tabbing}
    \textbf{procedure} lsolvetrans($t$, $L$,
       \textbf{var} $Y$, $X$);\\
    \textbf{if} $\sons(t)=\emptyset$ \textbf{then}\\
    \quad\= Solve $L|_{\hat t\times\hat t}^* X|_{\hat t} = Y|_{\hat t}$
        directly\\
    \textbf{else begin}\\
    \> lsolvetrans($t_2$, $L$, $Y$, $X$);\\
    \> addevaltrans($-1$, $t_2$, $t_1$, $L$, $Y$, $X$);\\
    \> lsolvetrans($t_1$, $L$, $Y$, $X$)\\
    \textbf{end}
  \end{tabbing}
  \hfill\strut
  \end{minipage}
  \vskip-\baselineskip

  \medskip
  
  \begin{minipage}{0.45\textwidth}
  \begin{tabbing}
    \textbf{procedure} rsolve($t$, $R$,
       \textbf{var} $Y$, $X$);\\
    \textbf{if} $\sons(t)=\emptyset$ \textbf{then}\\
    \quad\= Solve $R|_{\hat t\times\hat t} X|_{\hat t} = Y|_{\hat t}$
        directly\\
    \textbf{else begin}\\
    \> rsolve($t_2$, $R$, $Y$, $X$);\\
    \> addeval($-1$, $t_1$, $t_2$, $R$, $Y$, $X$);\\
    \> rsolve($t_1$, $R$, $Y$, $X$)\\
    \textbf{end}
  \end{tabbing}
  \hfill\strut
  \end{minipage}%
  \quad%
  \begin{minipage}{0.45\textwidth}
  \begin{tabbing}
    \textbf{procedure} rsolvetrans($t$, $R$,
       \textbf{var} $Y$, $X$);\\
    \textbf{if} $\sons(t)=\emptyset$ \textbf{then}\\
    \quad\= Solve $R|_{\hat t\times\hat t}^* X|_{\hat t} = Y|_{\hat t}$
        directly\\
    \textbf{else begin}\\
    \> rsolvetrans($t_1$, $R$, $Y$, $X$);\\
    \> addevaltrans($-1$, $t_1$, $t_2$, $R$, $Y$, $X$);\\
    \> rsolvetrans($t_2$, $R$, $Y$, $X$)\\
    \textbf{end}
  \end{tabbing}
  \hfill\strut
  \end{minipage}
  \vskip-\baselineskip
  \caption{Triangular solves
           $L|_{\hat t\times\hat t} X|_{\hat t} = Y|_{\hat t}$,
           $L|_{\hat t\times\hat t}^* X|_{\hat t} = Y|_{\hat t}$,
           $R|_{\hat t\times\hat t} X|_{\hat t} = Y|_{\hat t}$,
           $R|_{\hat t\times\hat t}^* X|_{\hat t} = Y|_{\hat t}$.}
  \label{fi:lsolve}
\end{figure}

In order to construct the LR factorization, we will also have
to solve the systems $LX=Y$ and $XR=Y$ with $\mathcal{H}$-matrices
$X$ and $Y$, and this requires some modifications to the
algorithms:
we consider the systems
\begin{align*}
  L|_{\hat t\times\hat t} X|_{\hat t\times\hat s}
  &= Y|_{\hat t\times\hat s}, &
  R|_{\hat t\times\hat t} X|_{\hat t\times\hat s}
  &= Y|_{\hat t\times\hat s}
\end{align*}
for blocks $(t,s)\in\ctII$.
On one hand, we can take advantage of the low-rank structure
if $(t,s)\in\lfaII$ holds, i.e., if $(t,s)$ is an admissible
leaf.
In this case, we have $Y|_{\hat t\times\hat s} = A_{Y,ts} B_{Y,ts}^*$
with $A_{Y,ts}\in\bbbr^{\hat t\times k}$ and
$B_{Y,ts}\in\bbbr^{\hat s\times k}$, and with the solution
$A_{X,ts}\in\bbbr^{\hat t\times k}$ of the linear system
\begin{equation*}
  L|_{\hat t\times\hat t} A_{X,ts} = A_{Y,ts},
\end{equation*}
we find that $X|_{\hat t\times\hat s} := A_{X,ts} B_{Y,ts}^*$ solves
\begin{equation*}
  L|_{\hat t\times\hat t} X|_{\hat t\times\hat s}
  = L|_{\hat t\times\hat t} A_{X,ts} B_{Y,ts}^*
  = A_{Y,ts} B_{Y,ts}^*
  = Y|_{\hat t\times\hat s}.
\end{equation*}
This property allows us to handle admissible blocks very
efficiently.

On the other hand, we cannot expect to be able to perform
the update $\widetilde{Y}_2 = Y|_{\hat t_2} - L_{21} X|_{\hat t_1}$
exactly, since we want to preserve the $\mathcal{H}$-matrix
structure of $Y$, so we have to use ``addmul'' instead of
``addeval'', approximating the intermediate result with the
given accuracy.

In order to keep the implementation simple, it also makes sense
to follow the structure of the block tree:
if we switch to the sons of $t$, we should also switch to the
sons of $s$, if it still has sons.
The resulting algorithms are summarized in Figure~\ref{fi:llsolve}.

%
%
\begin{figure}
  \begin{minipage}{0.45\textwidth}
  \begin{tabbing}
    \textbf{procedure} llsolve($t$, $s$, $L$,
    \textbf{var} $Y$, $X$);\\
    \textbf{if} $(t,s)\in\lfaII$ \textbf{then begin}\\
    \quad\= lsolve($t$, $L$, $A_{Y,ts}$, $A_{X,ts}$);\\
    \> $B_{X,ts} \gets B_{Y,ts}$\\
    \textbf{end else if} $(t,s)\in\lfiII$ \textbf{then}\\
    \> lsolve($t$, $L$, $N_{Y,ts}$, $N_{X,ts}$)\\
    \textbf{else for} $s'\in\sons(s)$ \textbf{ do begin}\\
    \> llsolve($t_1$, $s'$, $L$, $Y$, $X$);\\
    \> addmul($-1$, $t_2$, $t_1$, $s'$, $L$, $X$, $Y$);\\
    \> llsolve($t_2$, $s'$, $L$, $Y$, $X$)\\
    \textbf{end}
  \end{tabbing}
  \hfill\strut
  \end{minipage}%
  \quad%
  \begin{minipage}{0.45\textwidth}
  \begin{tabbing}
    \textbf{procedure} rlsolve($t$, $s$, $R$,
    \textbf{var} $Y$, $X$);\\
    \textbf{if} $(t,s)\in\lfaII$ \textbf{then begin}\\
    \quad\= rsolve($t$, $R$, $A_{Y,ts}$, $A_{X,ts}$);\\
    \> $B_{X,ts} \gets B_{Y,ts}$\\
    \textbf{end else if} $(t,s)\in\lfiII$ \textbf{then}\\
    \> rsolve($t$, $R$, $N_{Y,ts}$, $N_{X,ts}$)\\
    \textbf{else for} $s'\in\sons(s)$ \textbf{ do begin}\\
    \> rlsolve($t_2$, $s'$, $R$, $Y$, $X$);\\
    \> addmul($-1$, $t_1$, $t_2$, $s'$, $R$, $X$, $Y$);\\
    \> rlsolve($t_1$, $s'$, $R$, $Y$, $X$)\\
    \textbf{end}
  \end{tabbing}
  \hfill\strut
  \end{minipage}
  \vskip-\baselineskip
  \caption{Solving $L|_{\hat t\times\hat t} X|_{\hat t\times\hat s}
            = Y|_{\hat t\times\hat s}$
           and $R|_{\hat t\times\hat t} X|_{\hat t\times\hat s}
            = Y|_{\hat t\times\hat s}$ for $\mathcal{H}$-matrices $X$ and $Y$}
  \label{fi:llsolve}
\end{figure}

We also require counterparts for the systems
\begin{align*}
  X|_{\hat s\times\hat t} L|_{\hat t\times\hat t}
  &= Y|_{\hat s\times\hat t}, &
  X|_{\hat s\times\hat t} R|_{\hat t\times\hat t}
  &= Y|_{\hat s\times\hat t}
\end{align*}
for blocks $(s,t)\in\ctII$.
These can be constructed along the same lines as before and
are summarized in Figure~\ref{fi:lrsolve}.

%
%
\begin{figure}
  \begin{minipage}{0.45\textwidth}
  \begin{tabbing}
    \textbf{procedure} lrsolve($s$, $t$, $L$,
    \textbf{var} $Y$, $X$);\\
    \textbf{if} $(s,t)\in\lfaII$ \textbf{then begin}\\
    \quad\= lsolvetrans($t$, $L$, $B_{Y,st}$, $B_{X,st}$);\\
    \> $A_{X,st} \gets A_{Y,st}$\\
    \textbf{end else if} $(s,t)\in\lfiII$ \textbf{then}\\
    \> lsolvetrans($t$, $L$, $N_{Y,st}^*$, $N_{X,st}^*$)\\
    \textbf{else for} $s'\in\sons(s)$ \textbf{ do begin}\\
    \> lrsolve($s'$, $t_2$, $L$, $Y$, $X$);\\
    \> addmul($-1$, $s'$, $t_2$, $t_1$, $X$, $L$, $Y$);\\
    \> lrsolve($s'$, $t_1$, $L$, $Y$, $X$)\\
    \textbf{end}
  \end{tabbing}
  \hfill\strut
  \end{minipage}%
  \quad%
  \begin{minipage}{0.45\textwidth}
  \begin{tabbing}
    \textbf{procedure} rrsolve($s$, $t$, $R$,
    \textbf{var} $Y$, $X$);\\
    \textbf{if} $(s,t)\in\lfaII$ \textbf{then begin}\\
    \quad\= rsolvetrans($t$, $R$, $B_{Y,st}$, $B_{X,st}$);\\
    \> $A_{X,st} \gets A_{Y,st}$\\
    \textbf{end else if} $(s,t)\in\lfiII$ \textbf{then}\\
    \> rsolvetrans($t$, $R$, $N_{Y,st}^*$, $N_{X,st}^*$)\\
    \textbf{else for} $s'\in\sons(s)$ \textbf{ do begin}\\
    \> rrsolve($s'$, $t_1$, $R$, $Y$, $X$);\\
    \> addmul($-1$, $s'$, $t_1$, $t_2$, $X$, $R$, $Y$);\\
    \> rrsolve($s'$, $t_2$, $R$, $Y$, $X$)\\
    \textbf{end}
  \end{tabbing}
  \hfill\strut
  \end{minipage}
  \vskip-\baselineskip

  \caption{Solving $X|_{\hat s\times\hat t} L|_{\hat t\times\hat t} 
            = Y|_{\hat s\times\hat t}$
           and $X|_{\hat s\times\hat t} R|_{\hat t\times\hat t} 
            = Y|_{\hat s\times\hat t}$ for $\mathcal{H}$-matrices $X$ and $Y$}
  \label{fi:lrsolve}
\end{figure}

%
%

\paragraph{LR factorization}
Now we have the necessary tools at our disposal to address the
LR factorization.
Given an $\mathcal{H}$-matrix $G$, we consider computing
a lower triangular matrix $L$ with unit diagonal and an upper
triangular matrix $R$ with
\begin{equation*}
  G|_{\hat t\times\hat t} = L|_{\hat t\times\hat t} R|_{\hat t\times\hat t}
\end{equation*}
for a cluster $t\in\ctI$.
If $t$ is a leaf, $G|_{\hat t\times\hat t}$ is given in standard
array representation and we can compute the LR factorization by
the usual algorithms.

If $t$ is not a leaf, we follow (\ref{eq:submatrices_LR}) and define
\begin{align*}
  G_{\nu\mu} &:= G|_{\hat t_\nu\times\hat t_\mu} &
  &\text{ for all } \nu,\mu\in\{1,2\}.
\end{align*}
Our equation takes the form
\begin{align*}
  \begin{pmatrix}
    G_{11} & G_{12}\\
    G_{21} & G_{22}
  \end{pmatrix}
  &= G|_{\hat t\times\hat t}
   = L|_{\hat t\times\hat t} R|_{\hat t\times\hat t}
   = \begin{pmatrix}
    L_{11} & \\
    L_{21} & L_{22}
  \end{pmatrix}
  \begin{pmatrix}
    R_{11} & R_{12}\\
    & R_{22}
  \end{pmatrix}\\
  &= \begin{pmatrix}
    L_{11} R_{11} &
    L_{11} R_{12}\\
    L_{21} R_{11} &
    L_{21} R_{12} + L_{22} R_{22}
  \end{pmatrix}.
\end{align*}
We can compute the LR factorization of $G|_{\hat t\times\hat t}$ by first
computing the factorization $G_{11}=L_{11} R_{11}$ by recursion, followed
by solving $G_{21}=L_{21} R_{11}$ with ``rrsolve'' and
$G_{12}=L_{11} R_{12}$ with ``llsolve'', computing the
Schur complement $\widetilde{G}_{22}:=G_{22}-L_{21} R_{12}$ approximately
with ``addmul'', and finding the LR factorization
$L_{22} R_{22} = \widetilde{G}_{22}$, again by recursion.
The resulting algorithm is summarized in Figure~\ref{fi:lrdecomp},
where $G_{22}$ is overwritten by the Schur complement $\widetilde{G}_{22}$.

%
%
\begin{figure}
  \begin{tabbing}
    \textbf{procedure} lrdecomp($t$, \textbf{var} $G$, $L$, $R$);\\
    \textbf{if} $\sons(t)=\emptyset$ \textbf{then}\\
    \quad\= Compute $L|_{\hat t\times\hat t} R|_{\hat t\times\hat t} =
      G|_{\hat t\times\hat t}$ directly\\
    \textbf{else begin}\\
    \> lrdecomp($t_1$, $G$. $L$, $R$);\\
    \> llsolve($t_1$, $t_2$, $L$, $G$, $R$);\\
    \> rrsolve($t_2$, $t_1$, $R$, $G$, $L$);\\
    \> addmul($-1$, $t_2$, $t_1$, $t_2$, $L$, $R$, $G$);\\
    \> lrdecomp($t_2$, $G$, $L$, $R$)\\
    \textbf{end}\\
  \end{tabbing}
  \caption{Computing the LR factorization
           $G|_{\hat t\times\hat t}
            = L|_{\hat t\times\hat t} R|_{\hat t\times\hat t}$}
  \label{fi:lrdecomp}
\end{figure}

%
%

\paragraph{Triangular inversion}

The next operation to consider is the inversion of triangular
matrices, i.e., we are looking for
$\widetilde{L}:=L^{-1}$ and $\widetilde{R}:=R^{-1}$.
As before, we consider the inversion of submatrices
$L|_{\hat t\times\hat t}$ and $R|_{\hat t\times\hat t}$ for a cluster
$t\in\ctI$.

Again, if $t$ is a leaf, the matrices $L|_{\hat t\times\hat t}$ and
$R|_{\hat t\times\hat t}$ are given in standard representation and
can be inverted by standard algorithms.
If $t$ has sons, the inverses can be written as
\begin{align*}
  \begin{pmatrix}
    L_{11} & \\
    L_{21} & L_{22}
  \end{pmatrix}^{-1}
  &= \begin{pmatrix}
    L_{11}^{-1} & \\
    -L_{22}^{-1} L_{21} L_{11}^{-1} &
    L_{22}^{-1}
  \end{pmatrix},\\
  \begin{pmatrix}
    R_{11} & R_{12}\\
    & R_{22}
  \end{pmatrix}^{-1}
  &= \begin{pmatrix}
    R_{11}^{-1} & -R_{11}^{-1} R_{12} R_{22}^{-1}\\
    & R_{22}^{-1}
  \end{pmatrix},
\end{align*}
so we can compute the off-diagonal blocks by calling
``llsolve'' and ``lrsolve'' in the first case and
``rlsolve'' and ``rrsolve'' in the second case, and then invert
the diagonal blocks by recursion.
The algorithms are summarized in Figure~\ref{fi:linvert}.

%
%
\begin{figure}
  \begin{minipage}{0.45\textwidth}
  \begin{tabbing}
    \textbf{procedure} linvert($t$, $L$,
           \textbf{var} $\widetilde{L}$);\\
    \textbf{if} $\sons(t)=\emptyset$ \textbf{then}\\
    \quad\= Compute
      $\widetilde{L}|_{\hat t\times\hat t} = L|_{\hat t\times\hat t}^{-1}$
      directly\\
    \textbf{else begin}\\
    \> $\widetilde{L}|_{\hat t_2\times\hat t_1} \gets -L|_{\hat t_2\times\hat t_1}$;\\
    \> llsolve($t_2$, $t_1$, $L$, $\widetilde{L}$, $\widetilde{L}$);\\
    \> lrsolve($t_2$, $t_1$, $L$, $\widetilde{L}$, $\widetilde{L}$);\\
    \> linvert($t_1$, $L$, $\widetilde{L}$);\\
    \> linvert($t_2$, $L$, $\widetilde{L}$)\\
    \textbf{end}
  \end{tabbing}
  \hfill\strut
  \end{minipage}%
  \quad%
  \begin{minipage}{0.45\textwidth}
  \begin{tabbing}
    \textbf{procedure} rinvert($t$, $R$,
           \textbf{var} $\widetilde{R}$);\\
    \textbf{if} $\sons(t)=\emptyset$ \textbf{then}\\
    \quad\= Compute
      $\widetilde{R}|_{\hat t\times\hat t} = R|_{\hat t\times\hat t}^{-1}$
      directly\\
    \textbf{else begin}\\
    \> $\widetilde{R}|_{\hat t_1\times\hat t_2} \gets -R|_{\hat t_1\times\hat t_2}$;\\
    \> rlsolve($t_1$, $t_2$, $R$, $\widetilde{R}$, $\widetilde{R}$);\\
    \> rrsolve($t_1$, $t_2$, $R$, $\widetilde{R}$, $\widetilde{R}$);\\
    \> rinvert($t_1$, $R$, $\widetilde{R}$);\\
    \> rinvert($t_2$, $R$, $\widetilde{R}$)\\
    \textbf{enf}
  \end{tabbing}
  \hfill\strut
  \end{minipage}
  \vskip -\baselineskip

  \caption{Inverting triangular matrices
     $\widetilde{L}|_{\hat t\times\hat t} \gets L|_{\hat t\times\hat t}^{-1}$
     and
     $\widetilde{R}|_{\hat t\times\hat t} \gets R|_{\hat t\times\hat t}^{-1}$}
  \label{fi:linvert}
\end{figure}

%
%

\paragraph{Triangular matrix multiplication}
Finally, having $\widetilde{L}=L^{-1}$ and $\widetilde{R}=R^{-1}$
at our disposal, we consider computing the inverse
\begin{equation*}
  \widetilde{G} := G^{-1}
  = (LR)^{-1} = \widetilde{R} \widetilde{L}.
\end{equation*}
As in the previous cases, recursion leads to sub-problems of
the form
\begin{equation*}
  \widetilde{G}|_{\hat t\times\hat t}
  = \widetilde{R}|_{\hat t\times\hat t} \widetilde{L}|_{\hat t\times\hat t}
\end{equation*}
for clusters $t\in\ctI$.
If $t$ is a leaf, we can compute the product directly.

Otherwise, the equation takes the form
\begin{align*}
  \begin{pmatrix}
    \widetilde{R}_{11} & \widetilde{R}_{12}\\
    & \widetilde{R}_{22}
  \end{pmatrix}
  \begin{pmatrix}
    \widetilde{L}_{11} & \\
    \widetilde{L}_{21} & \widetilde{L}_{22}
  \end{pmatrix}
  &= \begin{pmatrix}
    \widetilde{R}_{11} \widetilde{L}_{11} +
    \widetilde{R}_{12} \widetilde{L}_{21} &
    \widetilde{R}_{12} \widetilde{L}_{22}\\
    \widetilde{R}_{22} \widetilde{L}_{21} &
    \widetilde{R}_{22} \widetilde{L}_{22}
  \end{pmatrix}.
\end{align*}
We can see that we have to compute products
$\widetilde{R}_{11} \widetilde{L}_{11}$ and
$\widetilde{R}_{22} \widetilde{L}_{22}$ that are of the same
kind as the original problem and can be handled by recursion.
We also have to compute the product
$\widetilde{R}_{12} \widetilde{L}_{21}$, which can be
accomplished by ``addmul''.

Finally, we have to compute $\widetilde{R}_{22} \widetilde{L}_{21}$
and $\widetilde{R}_{12} \widetilde{L}_{22}$, i.e., products of
triangular and non-triangular $\mathcal{H}$-matrices.
In order to handle this task, we could introduce suitable
counterparts of the algorithms ``rlsolve'' and
``lrsolve'' that multiply by a triangular matrix
instead of by its inverse.
These algorithms would in turn require counterparts of
``lsolve'' and ``rsolve'', i.e., we would
have to introduce four more algorithms.

To keep this article short, another approach can be used:
Since we have $\widetilde{R}_{22} = R_{22}^{-1}$ and
$\widetilde{L}_{22} = L_{22}^{-1}$, we can evaluate the product
$\widetilde{R}_{12} \widetilde{L}_{22}
= \widetilde{R}_{12} L_{22}^{-1}$ by the algorithm
``lrsolve'' and
$\widetilde{R}_{22} \widetilde{L}_{21}
= R_{22}^{-1} \widetilde{L}_{21}$ by the algorithm
``rlsolve'' without the need for additional algorithms.
The result is summarized in Figure~\ref{fi:lrinvert}.

%
%
\begin{figure}
  \begin{tabbing}
    \textbf{procedure} lrinvert($t$, $L$, $R$,
       \textbf{var} $\widetilde{L}$, $\widetilde{R}$, $\widetilde{G}$);\\
    \textbf{if} $\sons(t)=\emptyset$ \textbf{then}\\
    \quad\=  Compute $\widetilde{G}|_{\hat t\times\hat t}
          = R|_{\hat t\times\hat t}^{-1} L|_{\hat t\times\hat t}^{-1}$
      directly\\
    \textbf{else begin}\\
    \> lrinvert($t_1$, $L$, $R$, $\widetilde{L}$, $\widetilde{R}$,
              $\widetilde{G}$);\\
    \> addmul($1$, $t_1$, $t_2$, $t_1$,
              $\widetilde{R}$, $\widetilde{L}$, $\widetilde{G}$);\\
    \> lrsolve($t_1$, $t_2$,
              $L$, $\widetilde{R}$, $\widetilde{G}$);\\
    \> rlsolve($t_2$, $t_1$,
              $R$, $\widetilde{L}$, $\widetilde{G}$);\\
    \> lrinvert($t_2$, $L$, $R$, $\widetilde{L}$, $\widetilde{R}$,
                         $\widetilde{G}$)\\
    \textbf{end}
  \end{tabbing}

  \caption{Inversion using the LR factorization}
  \label{fi:lrinvert}
\end{figure}

%
%
\begin{remark}[In-place operation]
All algorithms for triangular matrices introduced in this
section can overwrite input variables with the result:
for triangular solves, the right-hand side can be overwritten
with the result, for the LR factorization, the lower and upper
triangular parts of the input matrix can be overwritten with
the triangular factors, the triangular inversion algorithms
can overwrite the input matrices with the inverses.

If we are only interested in computing the inverse, we can
interleave the algorithm ``lrinvert'' with ``linvert'' and
``rinvert'' to avoid additional storage for the intermediate
results $\widetilde{L}$ and $\widetilde{R}$:
once the LR factorization is available, we first overwrite the off-diagonal
blocks $L_{21}$ and $R_{12}$ by the intermediate results
$\widetilde{L}_{21} = -L_{22}^{-1} L_{21} L_{11}^{-1}$ and
$\widetilde{R}_{12} = -R_{11}^{-1} R_{12}^{-1} R_{22}^{-1}$, then overwrite
the first diagonal block recursively by its inverse, add the product
$\widetilde{R}_{12} \widetilde{L}_{21}$, then overwrite the no longer
required off-diagonal blocks with $R_{22}^{-1} \widetilde{L}_{21}$ and
$\widetilde{R}_{12} L_{22}^{-1}$.
A recursive call to compute the inverse of the second diagonal block
completes the algorithm.
\end{remark}


\section{Complexity estimates for combined operations}

Due to the lack of symmetry introduced by the low-rank
approximation steps required to compute an
$\mathcal{H}$-matrix, we cannot prove that the LR factorization
requires one third of the work of the matrix multiplication.
We can, however, prove that the LR factorization $G=LR$, the
inversion of the triangular factors, and the multiplication
$R^{-1} L^{-1}$ \emph{together} require not more work than
the matrix multiplication.

Before we can consider the $\mathcal{H}$-matrix case, we recall
the corresponding estimates for standard matrices, cf.
(\ref{eq:dense_operations}): the LR factorization,
triangular matrix inversion, and multiplication require
\begin{equation*}
  \frac{n}{6} (4n^2 - 3n - 1),\quad
  \frac{n}{6} (2n^2 + 4), \text{ and}\quad
  \frac{n}{6} (4 n^2 - 3n - 1) \quad\text{ operations.}
\end{equation*}
By adding the estimates for the four parts of the inversion
algorithm, we obtain a computational cost of
\begin{equation*}
  \frac{n}{6} (12 n^2 - 6 n + 6)
  = n (2n^2 - n + 1) \leq 2 n^3 \quad\text{ operations,}
\end{equation*}
i.e., inverting $G$ requires less operations than multiplying
the matrix by itself and adding the result to a matrix.
We aim to obtain a similar result for $\mathcal{H}$-matrices.

We assume that the block tree is \emph{admissible}, i.e., that
a leaf $b=(t,s)$ of the block tree $\ctII$ is either admissible or
has a leaf of $\ctI$ either as row or column cluster:
\begin{align}\label{eq:block_admissible}
  (t,s)\in\lfiII &\Rightarrow (t\in\lfI \vee s\in\lfI) &
  &\text{ for all } (t,s)\in\ctII,
\end{align}
and we assume that there is a constant $\varrho\in\bbbn$ such
that
\begin{align}\label{eq:resolution}
  t\in\lfI &\iff |\hat t|\leq\varrho &
  &\text{ for all } t\in\ctI.
\end{align}
The constant $\varrho$ is called the \emph{resolution} (sometimes
also the \emph{leaf size}) of $\ctI$.
Both properties (\ref{eq:block_admissible}) and (\ref{eq:resolution})
can be ensured during the construction of the cluster tree.

Now let us consider the number of operations for the algorithms
``lsolve'' and ``rsolve'' given in Figure~\ref{fi:lsolve}.
If $t\in\ctI$ is a leaf, solving the linear systems requires
$|\hat t|^2$ operations.
Otherwise, we just use ``addeval'' and recursive calls
and arrive at the recurrence formulas
\begin{align*}
  \Wls(t,\ell)
  &:= \begin{cases}
        \ell |\hat t|^2 &\text{ if } t\in\lfI,\\
        \Wls(t_1,\ell) + \Wls(t_2,\ell) + \Wev(t_2,t_1,\ell)
        &\text{ otherwise},
      \end{cases}\\
  \Wrs(t,\ell)
  &:= \begin{cases}
        \ell |\hat t|^2 &\text{ if } t\in\lfI,\\
        \Wrs(t_1,\ell) + \Wrs(t_2,\ell) + \Wev(t_1,t_2,\ell)
        &\text{ otherwise}
      \end{cases} &
  &\text{ for } t\in\ctI
\end{align*}
that give bounds for the number of operations required by
``lsolve'' and ``rsolve'', respectively,
where $\ell\in\bbbn$ again denotes the columns of the matrices
$X$ and $Y$.

%
%
\begin{lemma}[Solving linear systems]
\label{le:lsolve}
We have
\begin{align*}
  \Wls(t,\ell) + \Wrs(t,\ell)
  &\leq \Wev(t,t,\ell) &
  &\text{ for all } t\in\ctI,\ \ell\in\bbbn.
\end{align*}
\end{lemma}
\begin{proof}
By structural induction.

Let $t\in\lfI$.
We have $(t,t)\in\lfiII$ and therefore
\begin{equation*}
  \Wls(t,\ell) + \Wrs(t,\ell)
  = 2 \ell |\hat t|^2
  \leq \Wev(t,t,\ell).
\end{equation*}
Let now $t\in\ctI\setminus\lfI$ be such that our claim holds
for the sons $t_1$ and $t_2$.
We have
\begin{align*}
  \Wls(t,\ell) + \Wrs(t,\ell)
  &= \Wls(t_1,\ell) + \Wls(t_2,\ell) + \Wev(t_2,t_1,\ell)\\
  &+ \Wrs(t_1,\ell) + \Wrs(t_2,\ell) + \Wev(t_1,t_2,\ell)\\
  &\leq \Wev(t_1,t_1,\ell) + \Wev(t_2,t_2,\ell)
      + \Wev(t_2,t_1,\ell) + \Wev(t_1,t_2,\ell)\\
  &= \Wev(t,t,\ell).
\end{align*}
\end{proof}

In the next step, we compare the computational work for the
$\mathcal{H}$-matrix multiplication with that for the combination
of the algorithms ``llsolve'' and ``rlsolve'' or
``lrsolve'' and ``rrsolve'', respectively.

The computational work for the forward substitution algorithm
``llsolve'' given in Figure~\ref{fi:llsolve} can be bounded
by
\begin{align*}
  \Wll(t,s)
  &:= \begin{cases}
    \Wls(t,k) &\text{ if } (t,s)\in\lfaII,\\
    \Wls(t,|\hat s|) &\text{ if } (t,s)\in\lfiII,\\
    \sum_{s'\in\sons(s)} \Wll(t_1,s') + \Wll(t_2,s')
           + \Wmm(t_2,t_1,s') &\text{ otherwise}
  \end{cases}
\end{align*}
for all $(t,s)\in\ctII$, while we get
\begin{align*}
  \Wrl(t,s)
  &:= \begin{cases}
    \Wrs(t,k) &\text{ if } (t,s)\in\lfaII,\\
    \Wrs(t,|\hat s|) &\text{ if } (t,s)\in\lfiII,\\
    \sum_{s'\in\sons(s)} \Wrl(t_1,s') + \Wrl(t_2,s')
           + \Wmm(t_1,t_2,s') &\text{ otherwise}
  \end{cases}
\end{align*}
for all $(t,s)\in\ctII$ for the algorithm ``rlsolve''.

%
%
\begin{lemma}[Forward and backward solves]
\label{le:llsolve}
We have
\begin{align*}
  \Wll(t,s) + \Wrl(t,s) &\leq \Wmm(t,t,s) &
  &\text{ for all } (t,s)\in\ctII.
\end{align*}
\end{lemma}
\begin{proof}
By structural induction, where the base case $(t,s)\in\lfII$ is
split into two sub-cases for admissible and inadmissible leaves.

\noindent
\emph{Case 1:}
Let $(t,s)\in\lfaII$.
If $(t,t)\in\lfII$ holds, we have $t\in\lfI$ and
\begin{align*}
  \Wmm(t,t,s)
  &\geq \Wev(t,s,k_{tt})
   = \Wev(t,s,|\hat t|)
   \geq 2 |\hat t|^2 k\\
  &= \Wls(t,k) + \Wrs(t,k)
   = \Wll(t,s) + \Wrl(t,s).
\end{align*}
Otherwise, i.e., if $(t,t)\not\in\lfII$, we have $t\not\in\lfI$,
and Lemma~\ref{le:lsolve} yields
\begin{align*}
  \Wmm(t,t,s)
  &\geq \Wev(t,t,k_{ts})
   = \Wev(t,t,k)\\
  &\geq \Wls(t,k) + \Wrs(t,k)
   = \Wll(t,s) + \Wrl(t,s).
\end{align*}

\noindent
\emph{Case 2:}
Let $(t,s)\in\lfiII$.
If $(t,t)\in\lfII$ holds, we have $t\in\lfI$ and
\begin{align*}
  \Wmm(t,t,s)
  &\geq \Wev(t,s,k_{tt})
   = \Wev(t,s,|\hat t|)
   \geq 2 |\hat t|^2 |\hat s|\\
  &= \Wls(t,|\hat s|) + \Wrs(t,|\hat s|)
   = \Wll(t,s) + \Wrl(t,s).
\end{align*}
Otherwise, i.e., if $(t,t)\not\in\lfII$, we have $t\not\in\lfI$ and
therefore $s\in\lfI$.
Due to (\ref{eq:resolution}), this means $|\hat s| \leq \varrho < |\hat t|$,
and we can use $k_{ts} = |\hat s|$ and Lemma~\ref{le:lsolve} to obtain
\begin{align*}
  \Wmm(t,t,s)
  &\geq \Wev(t,t,k_{ts})
   = \Wev(t,t,|\hat s|)\\
  &\geq \Wls(t,|\hat s|) + \Wrs(t,|\hat s|)
   = \Wll(t,s) + \Wlr(t,s).
\end{align*}

\noindent
\emph{Case 3:}
Let $(t,s)\in\ctII\setminus\lfII$ be such that our claim holds
for all sons of $(t,s)$.
Since $(t,s)$ is not a leaf, we have $t\not\in\lfI$ and therefore
also $(t,t)\not\in\lfII$.
This implies
\begin{align*}
  \Wll(t,s) + \Wrl(t,s)
  &= \sum_{s'\in\sons(s)} \Wll(t_1,s') + \Wll(t_2,s') + \Wmm(t_2,t_1,s')\\
  &\qquad + \Wrl(t_1,s') + \Wrl(t_2,s') + \Wmm(t_1,t_2,s')\\
  &\leq \sum_{s'\in\sons(s)} \Wmm(t_1,t_1,s') + \Wmm(t_2,t_2,s')\\
  &\qquad + \Wmm(t_2,t_1,s') + \Wmm(t_1,t_2,s')\\
  &= \Wmm(t,t,s),
\end{align*}
and our proof is complete.
\end{proof}

Now we consider the two algorithms ``lrsolve'' and
``rrsolve''.
They rely on ``lsolvetrans'' and ``rsolvetrans'',
and these algorithms require the same work as ``lsolve'' and
``rsolve'', respectively.
The work for ``lrsolve'' and ``rrsolve'' is then
bounded by
\begin{align*}
  \Wlr(s,t)
  &= \begin{cases}
    \Wrs(t,k) &\text{ if } (s,t)\in\lfaII,\\
    \Wrs(t,|\hat s|) &\text{ if } (s,t)\in\lfiII,\\
    \sum_{s'\in\sons(s)} \Wlr(s',t_1) + \Wlr(s',t_2) + \Wmm(s',t_2,t_1)
    &\text{ otherwise},
  \end{cases}\\
  \Wrr(s,t)
  &= \begin{cases}
    \Wls(t,k) &\text{ if } (s,t)\in\lfaII,\\
    \Wls(t,|\hat s|) &\text{ if } (s,t)\in\lfiII,\\
    \sum_{s'\in\sons(s)} \Wrr(s',t_1) + \Wrr(s',t_2) + \Wmm(s',t_1,t_2)
    &\text{ otherwise}
  \end{cases}
\end{align*}
for all $(s,t)\in\ctII$.
Proceeding as in the proof of Lemma~\ref{le:llsolve} leads us to
\begin{align*}
  \Wlr(s,t) + \Wrr(s,t) &\leq \Wmm(s,t,t) &
  &\text{ for all } (s,t)\in\ctII.
\end{align*}
Now that the fundamental statements for the forward and backward
subsitution algorithms are at our disposal, we can consider the
factorization and inversion algorithms.
We directly obtain the bounds
\begin{align*}
  \Wdc(t)
  &:= \begin{cases}
    \frac{|\hat t|}{6} (4|\hat t|^2-3|\hat t|-1) &\text{ if } t\in\lfI,\\
    \Wdc(t_1) + \Wdc(t_2)
    + \Wll(t_1,t_2)\\
    + \Wrr(t_2,t_1) + \Wmm(t_2,t_1,t_2) &\text{ otherwise},
  \end{cases}\\
  \Wli(t)
  &:= \begin{cases}
    \frac{|\hat t|}{6} (2|\hat t|^2+4) &\text{ if } t\in\lfI,\\
    \Wll(t_2,t_1) + \Wlr(t_2,t_1)\\
    + \Wli(t_1) + \Wli(t_2) &\text{ otherwise},
  \end{cases}\\
  \Wri(t)
  &:= \begin{cases}
    \frac{|\hat t|}{6} (2|\hat t|^2+4) &\text{ if } t\in\lfI,\\
    \Wrl(t_1,t_2) + \Wrr(t_1,t_2)\\
    + \Wri(t_1) + \Wri(t_2) &\text{ otherwise},
  \end{cases}\\
  \Win(t)
  &:= \begin{cases}
    \frac{|\hat t|}{6} (4|\hat t|^2-3|\hat t|-1) &\text{ if } t\in\lfI,\\
    \Win(t_1) + \Win(t_2) + \Wmm(t_1,t_2,t_1)\\
    + \Wlr(t_1,t_2) + \Wrl(t_2,t_1) &\text{ otherwise}
  \end{cases}
\end{align*}
for all clusters $t\in\ctI$ and the algorithms ``lrdecomp'',
``linvert'', ``rinvert'', and ``lrinvert'', respectively.

%
%
\begin{theorem}[Combined complexity]
We have
\begin{align*}
  \Wdc(t) + \Wli(t) + \Wri(t) + \Win(t)
  &\leq \Wmm(t,t,t) &
  &\text{ for all } t\in\ctI.
\end{align*}
\end{theorem}
\begin{proof}
By structural induction.
We start with the base case $t\in\lfI$ and observe
\begin{align*}
  \Wdc(t) + \Wli(t) &+ \Wri(t) + \Win(t)\\
  &= \frac{|\hat t|}{6} (4|\hat t|^2-3|\hat t|-1 + 2|\hat t|^2+4 + 2|\hat t|^2+4 + 4|\hat t|^2-3|\hat t|-1)\\
  &= \frac{|\hat t|}{6} (12 |\hat t|^2 - 6 |\hat t| + 6)
   \leq 2 |\hat t|^3 = \Wev(t,t,|\hat t|) \leq \Wmm(t,t,t).
\end{align*}
Now let $t\in\ctI$ be chosen such that the estimate holds
for all of its sons.
We obtain
\begin{align*}
  \Wdc(t) + \Wli(t) &+ \Wri(t) + \Win(t)\\
  &= \Wdc(t_1) + \Wdc(t_2) + \Wll(t_1,t_2) + \Wrr(t_2,t_1)
     + \Wmm(t_2,t_1,t_2)\\
  &\quad + \Wll(t_2,t_1) + \Wlr(t_2,t_1) + \Wli(t_1) + \Wli(t_2)\\
  &\quad + \Wrl(t_1,t_2) + \Wrr(t_1,t_2) + \Wri(t_1) + \Wri(t_2)\\
  &\quad + \Win(t_1) + \Win(t_2) + \Wmm(t_1,t_2,t_1) + \Wlr(t_1,t_2)
     + \Wrl(t_2,t_1)\\
  &= \Wdc(t_1) + \Wli(t_1) + \Wri(t_1) + \Win(t_1)\\
  &\quad + \Wdc(t_2) + \Wli(t_2) + \Wri(t_2) + \Win(t_2)\\
  &\quad + \Wll(t_1,t_2) + \Wrl(t_1,t_2)\\
  &\quad + \Wlr(t_2,t_1) + \Wrr(t_2,t_1)\\
  &\quad + \Wll(t_2,t_1) + \Wrl(t_2,t_1)\\
  &\quad + \Wlr(t_1,t_2) + \Wrr(t_1,t_2)\\
  &\quad + \Wmm(t_2,t_1,t_2) + \Wmm(t_1,t_2,t_1)\\
  &\leq \Wmm(t_1,t_1,t_1) + \Wmm(t_2,t_2,t_2)\\
  &\quad + \Wmm(t_1,t_1,t_2) + \Wmm(t_2,t_1,t_1)\\
  &\quad + \Wmm(t_2,t_2,t_1) + \Wmm(t_1,t_2,t_2)\\
  &\quad + \Wmm(t_2,t_1,t_2) + \Wmm(t_1,t_2,t_1)
   \leq \Wmm(t,t,t),
\end{align*}
where we have used Lemma~\ref{le:llsolve} in the next-to-last
estimate.
\end{proof}


\section{\texorpdfstring{Complexity of the $\mathcal{H}$-matrix
  multiplication}{Complexity of the H-matrix multiplication}}

We have seen that the number of operations for our algorithms
can be bounded by the number of operations $\Wmm(t,s,r)$ required by the
matrix multiplication.
In order to complete the analysis, we derive a bound for
$\Wmm(t,s,r)$ that is sharper than the standard results provided
in \cite{GRHA02,HA15}.
We rely on the following assumptions:
\begin{itemize}
  \item for an inadmissible leaf $b=(t,s)\in\lfiII$ of the block
    tree, we have either $t\in\lfI$ or $s\in\lfI$,
  \item there is a constant $m\in\bbbn$, e.g., the resolution
    introduced in (\ref{eq:resolution}), such that
    \begin{align*}
      |\hat t| &\leq m &
      &\text{ for all leaves } t\in\lfI \text{ of the cluster tree},
    \end{align*}
  \item there is a constant $p\in\bbbn_0$ such that
    \begin{align*}
      \level(t) &\leq p &
      &\text{ for all } t\in\ctI,
    \end{align*}
    i.e., $p$ is an upper bound for the depth of the cluster tree,
  \item the block tree is \emph{sparse}, i.e., there is a constant
    $\Csp\in\bbbn$ such that
    \begin{subequations}\label{eq:Csp}
    \begin{align}
      |\{s\in\ctI\ :\ (t,s)\in\ctII\}| &\leq \Csp &
      &\text{ for all } t\in\ctI,\\
      |\{t\in\ctI\ :\ (t,s)\in\ctII\}| &\leq \Csp &
      &\text{ for all } s\in\ctI.
    \end{align}
    \end{subequations}
\end{itemize}
We denote the maximal rank of leaf blocks by
\begin{equation*}
  \hat k := \max\{ k_{ts}\ :\ (t,s)\in\lfII \} \leq \max\{ k, m \}.
\end{equation*}
In order to facilitate working with sums involving clusters and
blocks, we introduce the sets of \emph{descendants} of clusters
and blocks by
\begin{align*}
  \desc(t)
  &:= \{ t \}
    \cup \kern-5pt \bigcup_{t'\in\sons(t)} \kern-5pt \desc(t'),\\
  \desc(t,s)
  &:= \begin{cases}
    \{ (t,s) \}
    \cup \bigcup_{(t',s')\in\sons(t,s)}
          \desc(t',s') & \text{ if } (t,s)\in\ctII,\\
    \{ (t,s) \} &\text{ otherwise}
  \end{cases}
\end{align*}
for all $t,s\in\ctI$.
The special case $(t,s)\not\in\ctII$ will be important when
dealing with products of hierarchical matrices that are added
to a \emph{part} of a low-rank submatrix.

Since the matrix multiplication relies on the matrix-vector
multiplication, we start by deriving a bound for
$\Wev(t,s,\ell)$ introduced in (\ref{eq:Wev}).
If $(t,s)\in\lfiII$, our first assumption yields $t\in\lfI$ or
$s\in\lfI$.
In the first case, the second assumption gives us $|\hat t|\leq m$
and
\begin{equation*}
  \Wev(t,s,\ell) = \ell (2 |\hat t|\,|\hat s| + \min\{|\hat t|,|\hat s|\})
  \leq \ell (2 m |\hat s| + |\hat t|)
  \leq 2 \ell m (|\hat t| + |\hat s|).
\end{equation*}
In the second case, we have $|\hat s|\leq m$ and obtain
\begin{equation*}
  \Wev(t,s,\ell) = \ell (2 |\hat t|\,|\hat s| + \min\{|\hat t|,|\hat s|\})
  \leq \ell (2 m |\hat t| + |\hat s|)
  \leq 2 \ell m (|\hat t| + |\hat s|).
\end{equation*}
Now we can combine the estimates for the inadmissible leaves with
those for the admissible ones to find
\begin{align*}
  \Wev(t,s,\ell)
  &\leq \begin{cases}
     2 \ell \hat k (|\hat t| + |\hat s|)
     &\text{ if } (t,s)\in\lfII,\\
     \sum_{\substack{t'\in\sons(t)\\ s'\in\sons(s)}} \Wev(t',s',\ell)
     &\text{ otherwise}
  \end{cases}
\end{align*}
for all $b=(t,s)\in\ctII$.
A straightforward induction yields
\begin{align}\label{eq:Wev_tree}
  \Wev(t,s,\ell)
  &\leq 2 \ell \hat k \kern-5pt \sum_{(t',s')\in\desc(t,s)} \kern-5pt
     |\hat t'| + |\hat s'| &
  &\text{ for all } b=(t,s)\in\ctII.
\end{align}
Next we consider the low-rank update and look for an upper bound
for $\Wup(t,s)$ introduced in (\ref{eq:Wup}).
Using the same arguments as before, we find
\begin{align*}
  \Wup(t,s,\ell)
  &\leq \begin{cases}
     \sum_{\substack{t'\in\sons(t)\\ s'\in\sons(s)}}
       \Wup(t',s',\ell)
     &\text{ if } (t,s)\in\ctII\setminus\lfII,\\
     2 \ell \hat k (|\hat t|+|\hat s|)
     &\text{ if } (t,s)\in\lfiII,\\
     \Cad (\hat k+\ell)^2 (|\hat t|+|\hat s|)
     &\text{ otherwise}
  \end{cases}
\end{align*}
for all $t,s\in\ctI$, and introducing $\Cup := \max\{\Cad, 1\}$
yields the upper bound
\begin{align*}
  \Wup(t,s,\ell)
  &\leq \begin{cases}
     \sum_{\substack{t'\in\sons(t)\\ s'\in\sons(s)}}
       \Wup(t',s',\ell)
     &\text{ if } (t,s)\in\ctII\setminus\lfII,\\
     \Cup (\hat k+\ell)^2 (|\hat t|+|\hat s|)
     &\text{ otherwise}
  \end{cases}
\end{align*}
for all $t,s\in\ctII$ due to $2\ell \hat k\leq (\hat k+\ell)^2$.
A straightforward induction, keeping in mind the special case
of $\desc(t,s)$ for $(t,s)\not\in\ctII$, leads to the estimate
\begin{align}\label{eq:Wup_tree}
  \Wup(t,s,\ell)
  &\leq \Cup (\hat k+\ell)^2
    \kern-5pt \sum_{(t',s')\in\desc(t,s)} \kern-5pt
    |\hat t'| + |\hat s'| &
  &\text{ for all } t,s\in\ctII.
\end{align}
Now we can investigate the matrix multiplication, i.e., we can
look for an upper bound for $\Wmm(t,s,r)$ introduced
in (\ref{eq:Wmm}).
While the computational work for the matrix-vector multiplication
and the update depends only on two clusters $t,s\in\ctI$, the
matrix multiplication depends on three $t,s,r\in\ctI$.
We can collect these triples in a special \emph{product tree}
that represents the recursive structure of the algorithm.

%
%
\begin{definition}[Product tree]
\label{de:product_tree}
Given a cluster tree $\ctI$ and a corresponding block tree $\ctII$,
the \emph{product tree} $\ctIII$ is the minimal tree satisfying the
following conditions:
\begin{itemize}
  \item For every node $\pi\in\ctIII$ of the product tree,
    there are clusters $t,s,r\in\ctI$ with $\pi=(t,s,r)$.
  \item Let $t\in\ctI$ be the root of $\ctI$.
    Then $(t,t,t)$ is the root of $\ctIII$.
  \item A node $\pi=(t,s,r)\in\ctIII$ is a leaf if and only if
    $(t,s)\in\lfII$ or $(s,r)\in\lfII$.
    Otherwise, its sons are given by
    $\sons(\pi)=\sons(t)\times\sons(s)\times\sons(r)$.
\end{itemize}
\end{definition}
Due to this definition, $\pi=(t,s,r)\in\ctIII$ implies
$(t,s)\in\ctII$ and $(s,r)\in\ctII$.

Let $\pi=(t,s,r)\in\ctIII$.
If $(t,s)\in\lfII$, we can apply the estimates (\ref{eq:Wev_tree}) and
(\ref{eq:Wup_tree}) to obtain
\begin{equation*}
  \Wmm(t,s,r)
  \leq 2 \hat k^2 \kern-5pt \sum_{(s',r')\in\desc(s,r)} \kern-5pt
           (|\hat s|+|\hat r|)
       + 4 \Cup \hat k^2 \kern-5pt \sum_{(t',r')\in\desc(t,r)} \kern-5pt
           (|\hat t|+|\hat r|).
\end{equation*}
If $(s,r)\in\lfII$, we get
\begin{equation*}
  \Wmm(t,s,r)
  \leq 2 \hat k^2 \kern-5pt \sum_{(t',s')\in\desc(t,s)} \kern-5pt
           (|\hat t|+|\hat s|)
       + 4 \Cup \hat k^2 \kern-5pt \sum_{(t',r')\in\desc(t,r)} \kern-5pt
           (|\hat t|+|\hat r|).
\end{equation*}
Otherwise, we have
\begin{equation*}
  \Wmm(t,s,r)
  \leq \Cmg \hat k^2 (|\hat t|+|\hat r|)
    + \kern-5pt \sum_{\substack{t'\in\sons(t)\\ s'\in\sons(s)\\
                              r'\in\sons(r)}} \kern-5pt \Wmm(t',s',r').
\end{equation*}
In order to get rid of the recursion, we define the set of
descendants $\desc(t,s,r)$ for every triple $\pi=(t,s,r)\in\ctIII$
as before and obtain
\begin{subequations}\label{eq:Wmm_sums}
\begin{align}
  \Wmm(t,s,r)
  &\leq 2 \hat k^2 \kern-5pt \sum_{(t',s',r')\in\desc(t,s,r)}
          \sum_{(s'',r'')\in\desc(s',r')} \kern-5pt (|\hat s''|+|\hat r''|)
     \label{eq:Wmm_ev_sr}\\
  &\quad + 2 \hat k^2 \kern-5pt \sum_{(t',s',r')\in\desc(t,s,r)}
          \sum_{(t'',s'')\in\desc(t',s')} \kern-5pt (|\hat t''|+|\hat s''|)
     \label{eq:Wmm_ev_ts}\\
  &\quad + \Cup \hat k^2 \kern-5pt \sum_{(t',s',r')\in\desc(t,s,r)}
          \sum_{(t'',r'')\in\desc(t',r')} \kern-5pt (|\hat t''|+|\hat r''|)
     \label{eq:Wmm_up}\\
  &\quad + \Cmg \hat k^2 \kern-5pt \sum_{(t',s',r')\in\desc(t,s,r)}
          (|\hat t'|+|\hat r'|).
     \label{eq:Wmm_mg}
\end{align}
\end{subequations}
We can investigate each of these terms separately.
First we notice that Definition~\ref{de:cluster_tree} implies that
all index sets corresponding to clusters on the same level of $\ctI$
are disjoint, and we have
\begin{equation}\label{eq:cluster_bound}
  \sum_{t'\in\desc(t)} |\hat t'|
  = \sum_{\ell=0}^p
    \sum_{\substack{t'\in\desc(t)\\ \level(t')=\ell}} |\hat t'|
  \leq \sum_{\ell=0}^p |\hat t| = (p+1) |\hat t|
\end{equation}
for all $t\in\ctI$.
Next we notice that Definition~\ref{de:block_tree} and the sparsity
assumption (\ref{eq:Csp}) together with (\ref{eq:cluster_bound}) yield
\begin{subequations}\label{eq:block_bound}
\begin{align}
  \sum_{(t',s')\in\desc(t,s)} |\hat t'|
  &= \sum_{t'\in\desc(t)} \sum_{\substack{s'\in\desc(s)\\ (t',s')\in\ctII}}
        |\hat t'|
   \leq \Csp \sum_{t'\in\desc(t)} |\hat t'|
   \leq \Csp (p+1) |\hat t|,\\
  \sum_{(t',s')\in\desc(t,s)} |\hat s'|
  &= \sum_{s'\in\desc(s)} \sum_{\substack{t'\in\desc(t)\\ (t',s')\in\ctII}}
        |\hat s'|
   \leq \Csp \sum_{s'\in\desc(s)} |\hat s'|
   \leq \Csp (p+1) |\hat s|
\end{align}
\end{subequations}
for all $(t,s)\in\ctII$.
Finally we observe that Definition~\ref{de:product_tree} ensures
that $(t',s',r')\in\ctIII$ implies both $(t',s')\in\ctII$ and
$(s',r')\in\ctII$, so that we can again use the sparsity
assumption (\ref{eq:Csp}) together with (\ref{eq:block_bound}) to get
\begin{subequations}\label{eq:product_bound}
\begin{align}
  \sum_{(t',s',r')\in\desc(t,s,r)} |\hat t'|
  &\leq \sum_{(t',s')\in\desc(t,s)}
     \sum_{\substack{r'\in\ctI\\ (s',r')\in\ctII}} |\hat t'|
   \leq \Csp^2 (p+1) |\hat t|,\\
  \sum_{(t',s',r')\in\desc(t,s,r)} |\hat s'|
  &\leq \sum_{(t',s')\in\desc(t,s)}
     \sum_{\substack{r'\in\ctI\\ (s',r')\in\ctII}} |\hat s'|
   \leq \Csp^2 (p+1) |\hat s|,\\
  \sum_{(t',s',r')\in\desc(t,s,r)} |\hat r'|
  &\leq \sum_{(s',r')\in\desc(s,r)}
     \sum_{\substack{t'\in\ctI\\ (t',s')\in\ctII}} |\hat r'|
   \leq \Csp^2 (p+1) |\hat r|
\end{align}
\end{subequations}
for all $(t,s,r)\in\ctII$.
With these preliminary estimates at our disposal, we can consider
the sums appearing in (\ref{eq:Wmm_sums}).

For (\ref{eq:Wmm_ev_sr}), we can take advantage of the fact that
due to (\ref{eq:Csp}), for every $(s',r')\in\ctII$, there are at most
$\Csp$ clusters $t'\in\ctI$ such that $(t',s',r')\in\ctII$, so we find
\begin{align*}
  \sum_{(t',s',r')\in\desc(t,s,r)} &\sum_{(s'',r'')\in\desc(s',r')}
     (|\hat s''|+|\hat r''|)\\
  &\leq \Csp \sum_{(s',r')\in\desc(s,r)} \sum_{(s'',r'')\in\desc(s',r')}
     (|\hat s''|+|\hat r''|)\\
  &= \Csp \sum_{(s'',r'')\in\desc(s,r)}
     \sum_{\substack{(s',r')\in\desc(s,r)\\ (s'',r'')\in\desc(s',r')}}
     (|\hat s''|+|\hat r''|).
\end{align*}
By Definition~\ref{de:block_tree}, the depth of the block tree $\ctII$ is
bounded by the depth of the cluster tree $\ctI$, and therefore by the
constant $p$ introduced in our assumptions.
Since every block has at most one father, a block $(s'',r'')\in\desc(s,r)$
cannot have more than $p+1$ predecessors $(s',r')\in\desc(s,r)$,
and we can use (\ref{eq:block_bound}) to get
\begin{align*}
  \sum_{(t',s',r')\in\desc(t,s,r)} &\sum_{(s'',r'')\in\desc(s',r')}
     (|\hat s''|+|\hat r''|)\\
  &\leq \Csp (p+1) \sum_{(s'',r'')\in\desc(s,r)}
     (|\hat s''|+|\hat r''|)\\
  &\leq \Csp^2 (p+1)^2 (|\hat s| + |\hat r|).
\end{align*}
We can proceed in a similar manner for (\ref{eq:Wmm_ev_ts}) to
get
\begin{equation*}
  \sum_{(t',s',r')\in\desc(t,s,r)} \sum_{(s'',r'')\in\desc(s',r')}
     (|\hat s''|+|\hat r''|)
  \leq \Csp^2 (p+1)^2 (|\hat t| + |\hat s|).
\end{equation*}
The sum (\ref{eq:Wmm_mg}) can be handled using (\ref{eq:product_bound})
to get
\begin{equation*}
  \sum_{(t',s',r')\in\desc(t,s,r)} (|\hat t'| + |\hat r'|)
  \leq \Csp^2 (p+1) (|\hat t| + |\hat r|).
\end{equation*}
This leaves us with only (\ref{eq:Wmm_up}).
Here, we have to distinguish two cases:
if $(t',r')\in\ctII$, we can proceed as in the first two cases to
find
\begin{align*}
  \sum_{\substack{(t',s',r')\in\desc(t,s,r)\\ (t',r')\in\ctII}}
  &\sum_{(t'',r'')\in\desc(t',r')}
    (|\hat t''| + |\hat r''|)\\
  &\leq \Csp \sum_{(t',r')\in\desc(t,r)} \sum_{(t'',r'')\in\desc(t',r')}
    (|\hat t''| + |\hat r''|)\\
  &= \Csp \sum_{(t'',r'')\in\desc(t,r)}
     \sum_{\substack{(t',r')\in\desc(t,r)\\ (t'',r'')\in\desc(t',r')}}
    (|\hat t''| + |\hat r''|)\\
  &\leq \Csp (p+1) \sum_{(t'',r'')\in\desc(t,r)}
    (|\hat t''| + |\hat r''|)\\
  &\leq \Csp^2 (p+1)^2 (|\hat t| + |\hat r|),
\end{align*}
where we have used (\ref{eq:block_bound}) in the last step.
If $(t',r')\not\in\ctII$, i.e., if it has been necessary to split
an admissible leaf block temporarily, the definition of
$\desc(t',r')$ yields
\begin{align*}
  \sum_{\substack{(t',s',r')\in\desc(t,s,r)\\ (t',r')\not\in\ctII}}
  &\sum_{(t'',r'')\in\desc(t',r')}
    (|\hat t''| + |\hat r''|)\\
  &= \sum_{\substack{(t',s',r')\in\desc(t,s,r)\\ (t',r')\not\in\ctII}}
    (|\hat t'| + |\hat r'|)\\
  &\leq \sum_{(t',s',r')\in\desc(t,s,r)} (|\hat t'| + |\hat r'|)\\
  &\leq \Csp^2 (p+1) (|\hat t| + |\hat r|)
\end{align*}
using (\ref{eq:product_bound}) in the last step.
Collecting all parts of the sum (\ref{eq:Wmm_sums}) yields
our final result.

%
%
\begin{theorem}[Matrix multiplication]
Let $(t,s,r)\in\ctIII$.
We have
\begin{equation*}
  \Wmm(t,s,r)
  \leq \Cmm \Csp^2 (p+1)^2 \hat k^2 (|\hat t| + |\hat s| + |\hat r|)
\end{equation*}
with $\Cmm := 4 + 2 \Cup + \Cmg$.
\end{theorem}

\bibliographystyle{plain}
\bibliography{hmatrix}

\end{document}